\documentclass[11pt, a4paper, reqno]{amsart}

\usepackage{anysize}
\usepackage[utf8]{inputenc}
\usepackage{amsmath}
\usepackage{mathtools}
\usepackage{amssymb}
\usepackage{hyperref}
\usepackage{thmtools, thm-restate}
\usepackage{xcolor}
\usepackage{mdframed}
\usepackage{comment}

\marginsize{2.5cm}{2.5cm}{3cm}{3cm}
\newtheorem{theorem}{Theorem}

\newtheorem{lemma}{Lemma}
\newtheorem{proposition}{Proposition}

\theoremstyle{definition}
\newtheorem{definition}{Definition}

\newcommand{\C}{\mathbb{C}}      		    	% complex numbers
\newcommand{\N}{\mathbb{N}}		            	% natural numbers
\newcommand{\R}{\mathbb{R}}		            	% real numbers
\renewcommand{\d}{\mathrm{d}}                   % differential

\newcommand{\p}{\mathfrak{p}} % function property

\newcommand{\q}{\mathfrak{q}} % function property

\newcommand{\defeq}{\coloneqq}              			% definition-equals sign
\newcommand{\eref}[1]{Eq.~(\ref{#1})}           			% equation reference

\newcommand{\neref}[1]{(\ref{#1})}           			% equation reference

\newcommand{\sumvariable}{\mathcal{Z}}

\newcommand{\thmref}[1]{Theorem~\ref{#1}}   		% theorem reference
\newcommand{\thmrefs}[2]{Theorems~\ref{#1}~and~\ref{#2}}   		% theorem reference
\newcommand{\lemref}[1]{Lemma~\ref{#1}}     		% lemma reference
\newcommand{\propref}[1]{Proposition~\ref{#1}}     		% proposition reference

\begin{document}

\title[On a Weighted Series of the Hurwitz Zeta Function]{On a Weighted Series of the Hurwitz Zeta Function}

\author{Matthew Fox}
\address{Department of Physics\\ University of California, Santa Barbara}
\email{msfox@ucsb.edu}

\author{Chaitanya Karamchedu}
\address{Department of Computer Science\\ University of Maryland}
\email{cdkaram@umd.edu}

\maketitle

\begin{abstract}
In this note we prove that for all $a \in \N$, $x \in \R_+ \cup \{0\}$, and $s \in \C$ with $\Re(s) > a + 2$, the (alternating) weighted series of the Hurwitz zeta function,
$$
\sum_{k \geq 1} (\pm 1)^k (k + x)^a\zeta(s,k + x),
$$
resolves into a finite combination of Hurwitz (Lerch) zeta functions. This applies in Marichal and Zena\"idi's theory on analogues of the Bohr-Mollerup theorem for higher-order convex functions.
\end{abstract}

\section{Introduction}

The Bohr-Mollerup theorem proves that the only logarithmically convex functions $f : \R_+ \rightarrow \R_+$ that satisfy $f(x+1) = xf(x)$ are the positive constant multiples of the gamma function, $\Gamma(x)$. The equivalent, but additive formulation of this theorem is the following, wherein we write $\Delta$ for the unit-step forward difference operator $\Delta : f(x) \mapsto f(x + 1) - f(x)$.
\begin{theorem}[Bohr-Mollerup \cite{Artin1, Bohr1}]
\label{thm:Bohr-Mollerup}
All convex solutions $f : \R_+ \rightarrow \R$ to the difference equation $\Delta f(x) = \ln x$ are of the form $f(x) = c + \ln \Gamma(x)$ for some constant $c \in \R$.
\end{theorem}

It is natural to consider analogues of this theorem for different difference equations as well as alternative function properties $\p$ such as monotonicity or higher-order convexity. Indeed, given function properties $\p$ and $\q$ as well as a function $g : \R_+ \rightarrow \R$ satisfying $\p$, it is evidently in the spirit of the Bohr-Mollerup theorem to ask: Does there exist a function $f : \R_+ \rightarrow \R$ satisfying $\q$ and the difference equation $\Delta f(x) = g(x)$ for all $x \in \R_+$? If so, is $f$ unique modulo an additive constant? And, if so, does $g$ completely determine $f$?

Krull was first to address such questions for eventually convex (concave) functions, which are functions that are convex (concave) on a neighborhood of infinity. Indeed, Krull proves that for any eventually convex (concave) function $g : \R_+ \rightarrow \R$ satisfying, for all integers $p > 0$,
\begin{equation}
\lim_{x \rightarrow \infty}\Delta^p g(x) = 0,
\label{eq:asymptoticcondition}
\end{equation}
there exists a unique, eventually concave (convex) solution $f : \R_+ \rightarrow \R$ to the difference equation $\Delta f(x) = g(x)$, modulo an additive constant \cite{Krull1, Krull2}. In addition, Krull also establishes that $g$ completely determines $f$, and even provides the exact functional relationship. Later, but independently, Webster established the multiplicative version of Krull's results \cite{Webster2, Webster1}.

While Krull and Webster's work constitute a considerable generalization of the Bohr-Mollerup theorem to a broader class of functions, the asymptotic condition \neref{eq:asymptoticcondition} is extremely restrictive and is not satisfied by even simple functions like $g(x) = x\ln x$ or apposite functions like $g(x) = \ln \Gamma(x)$ (see \thmref{thm:Bohr-Mollerup}). Thus, one naturally seeks to weaken it.

In their monograph \cite{Marichal1}, Marichal and Zena\"idi do exactly this, and in consequence extend the Bohr-Mollerup theorem to an even broader class of functions. Specifically, their results envelop the real-valued, eventually $p$-convex ($p$-concave) functions for every $p \in \N \cup \{-1\}$, where $\N$ includes $0$. Such ``higher-order'' convex and concave functions naturally appear in Newton interpolation theory \cite{Boor1}; they are defined as follows. 

\begin{definition}
Fix $p \in \N \cup \{-1\}$ and let $I \subseteq \R_+$ be an open and convex set. A function $f : \R_+ \rightarrow \R$ is \emph{$p$-convex} (\emph{$p$-concave}) on $I$ if and only if for all systems $x_0 < x_1 < \dots < x_{p + 1}$ of $p + 2$ points in $I$, the divided difference of $f$ at $x_0, x_1, \dots, x_{p+1}$ is non-negative (non-positive) \cite{Kuczma1, Marichal1}. Thus, a \emph{$1$-convex} (\emph{$1$-concave}) function is an ordinary convex (concave) function, a \emph{$0$-convex} (\emph{$0$-concave}) function is a weakly monotone increasing (decreasing) function, and a \emph{$-1$-convex} (\emph{$-1$-concave}) function is a non-negative (non-positive) function. Additionally, we say $f$ is \emph{eventually $p$-convex} (\emph{eventually $p$-concave}) if and only if there exists a convex neighborhood of infinity $I \subseteq \R_+$ on which $f$ is $p$-convex ($p$-concave).
\end{definition}

We now state Marichal and Zena\"idi's main theorem on analogues of the Bohr-Mollerup theorem for higher-order convex and concave functions.
\begin{theorem}[Marichal-Zena\"idi \cite{Marichal1}]
\label{thm:maintheorem}
Fix $p \in \N$ and suppose that the function $g : \R_+ \rightarrow \R$ is eventually $p$-convex ($p$-concave) and has the asymptotic property that the sequence $n \mapsto \Delta^p g(n)$ converges to zero. Then there exists a unique, eventually $p$-concave ($p$-convex) solution $f : \R_+ \rightarrow \R$ to the difference equation $\Delta f(x) = g(x)$, modulo an additive constant $c \in \R$. Moreover, $g$ uniquely determines $f$ through the equation
\begin{equation}
f(x) = c -g(x) + \lim_{n \rightarrow \infty} \left[ \sum_{k=1}^{n-1} \big(g(k) - g(k + x)\big) + \sum_{j = 1}^p \binom{x}{j}\Delta^{j-1} g(n) \right].
\label{eq:uniquesolution}
\end{equation}
\end{theorem}

In expounding their theory, Marichal and Zena\"idi compose a compendium of analogues of \thmref{thm:Bohr-Mollerup} for a variety of important functions, including offshoots of the gamma function like the polygamma, $q$-gamma, and Barnes $G$ functions, but also, interestingly, more disparate functions like the Catalan numbers function, the generalized Stieltjes constants, and the Hurwitz zeta function. All of these analogues follow from \thmref{thm:maintheorem}.

In particular, for the Hurwitz zeta function
$$
\zeta(s,x) \defeq \sum_{n \geq 0} \frac{1}{(n + x)^s}
$$
with real $s > 1$, they establish the following uniqueness result.
\begin{proposition}[Marichal-Zena\"idi \cite{Marichal1}]
\label{prop:HurwitzZetaUniqueness}
For all real $s > 1$, all eventually monotone (that is, $0$-convex or $0$-concave) solutions $f_s : \R_+ \rightarrow \R$ to the difference equation $\Delta f_s(x) = -x^{-s}$ are of the form $f_s(x) = c_s + \zeta(s,x)$ for some constant $c_s \in \R$.
\end{proposition}

Together with the function property of eventual monotonicity, they also establish a reciprocal result wherein instead of characterizing $\zeta(s,x)$ as the unique, eventually monotone \emph{solution} to a particular difference equation, they characterize the unique, eventually monotone solutions $f_s$ to the difference equation that \emph{evaluates} to $\zeta(s,x)$, namely $\Delta f_s(x) = \zeta(s,x)$.
\begin{proposition}[Marichal-Zena\"idi \cite{Marichal1}]
\label{prop:marichalzenaidizetaprop}
For all real $s > 2$, all eventually monotone solutions $f_s : \R_+ \rightarrow \R$ to the difference equation $\Delta f_s(x) = \zeta(s,x)$ are of the form 
\begin{equation}
f_s(x) = c_s + \zeta(s - 1) - \sum_{k \geq 0}\zeta(s, k + x)
\label{eq:marichalpropfunction}
\end{equation}
for some constant $c_s \in \R$.
\end{proposition}

Incidentally, since for every real $s > 1$ the Hurwitz zeta function $\zeta(s,x)$ is convex in $x$ on $\R_+$, it is not hard to prove the following generalization of \propref{prop:marichalzenaidizetaprop}.
\begin{proposition}
\label{prop:marichalzenaidizetaprop2}
For all real $s > 2$ and $p \in \{0,1\}$, all eventually $p$-concave solutions $f_s : \R_+ \rightarrow \R$ to the difference equation $\Delta f_s(x) = \zeta(s,x)$ are of the form \neref{eq:marichalpropfunction}.
\end{proposition}

As a corollary to Propositions~\ref{prop:HurwitzZetaUniqueness} and \ref{prop:marichalzenaidizetaprop}, Marichal and Zena\"idi derive a myriad of intriguing identities involving the Hurwitz zeta function, including
\begin{equation}
\sum_{k \geq 1} \zeta(s, k) = \zeta(s - 1).
\label{eq:zetaidentity}
\end{equation}
We note that while Marichal and Zena\"idi obtain \eref{eq:zetaidentity} assuming real $s > 2$, their identity readily extends to complex $s$ in the connected open set $\Omega_\epsilon \defeq \{s \in \C : \Re(s) > 2 + \epsilon\}$ for any fixed $\epsilon > 0$. Indeed, since $\zeta(s-1)$ and the partial sums $\big\{\sum_{k=1}^n \zeta(s,k) : n \in \N\big\}$ are holomorphic on $\Omega_\epsilon$ for any fixed $\epsilon > 0$, and since 
\begin{equation}
\sup_{s \in \Omega_\epsilon}\left| \sum_{k = 1}^n \zeta(s,k)\right| < \sum_{k = 1}^n \zeta(2 + \epsilon,k)
\label{eq:Mtest}
\end{equation}
for every $n \in \N$, together \eref{eq:zetaidentity} and the Weierstrass M-test imply $\sum_{k \geq 1}\zeta(s,k)$ converges uniformly on $\Omega_\epsilon$. Consequently, $\sum_{k \geq 1}\zeta(s,k)$ is holomorphic on $\Omega_\epsilon$, and so, by the identity theorem, \eref{eq:zetaidentity} holds for all $s \in \Omega_\epsilon$ for any fixed $\epsilon > 0$.

Of course, it is natural to wonder if this argument can accommodate $\epsilon = 0$. However, without a stronger bound in \eref{eq:Mtest}, the holomorphy of $\sum_{k \geq 1}\zeta(s,k)$ does not readily extend. In \cite{Paris1}, however, Paris circumvents this issue by explicitly evaluating the LHS of \eref{eq:zetaidentity} via the integral representation of $\zeta(s,x)$. Ultimately, he proves that \eref{eq:zetaidentity} holds on $\Omega_\epsilon$ for all $\epsilon \geq 0$, as desired.

In the same paper \cite{Paris1}, Paris also investigates the more general series
\begin{equation}
\sum_{k \geq 1}(\pm1)^k k^a \zeta(s, k + x)
\label{eq:parissum}
\end{equation}
for $a \in \N$, $x \in \R_+ \cup \{0\}$, and $s \in \C$ with $\Re(s) > a + 2$ in the non-alternating ($+$) case and $\Re(s) > a + 1$ in the alternating ($-$) case.

With $a = 0$ in the non-alternating case, Paris' sum \neref{eq:parissum} appears in \eref{eq:marichalpropfunction}, and in particular controls the functional dependence on $x$ of the solutions $f_s(x)$ in Propositions~\ref{prop:marichalzenaidizetaprop} and \ref{prop:marichalzenaidizetaprop2}. By his same method of explicitly evaluating the sum via the integral representation of $\zeta(s,x)$, Paris solves \eref{eq:parissum} exactly for $x = 0$ and $a \in \{0,1,2,3\}$ in the non-alternating case, and for $x 
= 0$ and $a \in \{0,1,2\}$ in the alternating case. In each evaluation, he finds that the series equals some finite combination of Hurwitz zeta functions. This begs the question if that pattern holds in general, for all $a \in \N$ and $x \in \R_+ \cup \{0\}$.

To address this question, we consider the following generalization of Paris' sum,
\begin{equation}
\sumvariable_\pm(s,a,x) \defeq \sum_{k \geq 1} (\pm 1)^k (k + x)^a\zeta(s, k + x),
\label{eq:phisum}
\end{equation}
where $a \in \N$, $x \in \R_+ \cup \{0\}$, and $s \in \C$ with $\Re(s) > a + 2$. Note that our sum subsumes Paris' sum \neref{eq:parissum} whenever ours absolutely converges, because in this case the expression $(k + x)^a$ can be expanded into a finite combination of Paris' sum, which can, in turn, be recursively simplified. Thus, a closed-form for $\sumvariable_\pm(s,a,x)$ implies a closed-form for \eref{eq:parissum}.

It is the purpose of this note to solve $\sumvariable_\pm(s,a,x)$ exactly for all $a \in \N$, $x \in \R_+ \cup \{0\}$, and $s \in \C$ with $\Re(s) > a + 2$ in both the non-alternating and alternating cases. In particular, in line with the pattern that Paris observed for \neref{eq:parissum}, we show that $\sumvariable_\pm(s,a,x)$ can always be resolved into a finite combination of Hurwitz or Lerch zeta functions, depending on the alternation. Denoting by $B_a$ the $a$th Bernoulli number (for which we adopt the convention $B_1 = +\frac{1}{2}$), $B_a(x)$ the $a$th Bernoulli polynomial, $E_a$ the $a$th Euler number, $E_a(x)$ the $a$th Euler polynomial, and $L(\lambda, s, x)$ the Lerch zeta function,
\begin{equation}
L(\lambda, s, x) \defeq \sum_{n \geq 0} \frac{\exp(2\pi i \lambda n)}{(n + x)^s},
\label{eq:Lerch}
\end{equation}
our main results are as follows.
\begin{theorem}
\label{thm:nonalternatingcase}
If $a \in \N$ and $s \in \C$ with $\Re(s) > a+2$, then for all $x \in \R_+ \cup \{0\}$,
$$
\sumvariable_+(s, a,x) = -\frac{1}{a + 1} B_{a + 1}(x + 1) \zeta(s, x+1) + \frac{1}{a + 1}\sum_{j=0}^{a+1} \binom{a + 1}{j} B_j \zeta(s-a+j-1,x+1).
$$
\end{theorem}

\begin{theorem}
\label{thm:alternatingcase}
If $a \in \N$ and $s \in \C$ with $\Re(s) > a+2$, then for all $x \in \R_+ \cup \{0\},$
$$
\sumvariable_{-}(s,a, x) = - \frac{1}{2}E_a(x+1)\zeta(s,x+1) - \sum_{j=0}^{a} \sum_{\ell = 0}^{a-j} \frac{E_j}{2^{a-\ell + 1}}\binom{a}{j} \binom{a - j}{\ell} L(-1, s - \ell, x+1).
$$
\end{theorem} 
Note that while $\sumvariable_{-}(s,a,x)$ converges when $\Re(s) > a + 1$, our work leaves open how to simplify $\sumvariable_{-}(s,a,x)$ for $s \in \C$ with $\Re(s) \in (a + 1, a + 2]$.

Interestingly, if $x = 0$, then the form of $\sumvariable_+(s,a,x)$ simplifies considerably; in particular, it resolves into a finite combination of Riemann zeta functions:
$$
\sumvariable_+(s,a,0) = \sum_{k \geq 1}k^a\zeta(s,k) = \frac{1}{a+1}\sum_{j = 0}^{a} \binom{a+1}{j} B_j \zeta(s-a+j-1).
$$
Importantly, this identity and the more general non-alternating sum $\sumvariable_+(s,a,x)$ are relevant in Marichal and Zena\"idi's theory on analogues of the Bohr-Mollerup theorem, as the next proposition shows.
\begin{proposition}
\label{prop:phibohrpolynomial}
Let $Q(x) = \omega_0 + \omega_1x + \cdots + \omega_\ell x^\ell$ be a real, degree-$\ell$ polynomial with $\omega_\ell \neq 0$. If $\omega_\ell > 0$ ($\omega_\ell < 0$), then for all real $s > \ell + 2$ and $p \in \{0,1\}$, all eventually $p$-concave ($p$-convex) solutions $f_{s,Q} : \R_+ \rightarrow \R$ to the difference equation $\Delta f_{s,Q}(x) = Q(x)\zeta(s,x)$ are of the form
$$
f_{s,Q}(x) = c_{s,Q} - \sum_{a = 0}^\ell \omega_a x^a\zeta(s,x) + \sum_{a = 0}^\ell\omega_a\sumvariable_+(s,a,0) - \sum_{a=0}^\ell\omega_a\sumvariable_+(s,a,x)
$$
for some constant $c_{s,Q} \in \R$.
\end{proposition}

Of course, there exist many more generalizations of Paris' sum, including in particular
\begin{equation}
\sum_{k \geq 1} f(k,x) \zeta(s,g(k,x))
\label{eq:moregeneral}
\end{equation}
for functions $f,g : \N \times \R \rightarrow \R$. As noted by Paris \cite{Paris1}, sums of this form bear a discretized resemblance to integrals like $\int_\Omega f(\alpha)\zeta(s,g(\alpha))\, \d\alpha,$ which were studied extensively in \cite{Espinosa1, Espinosa2} for $\Omega = (0,1)$ and various $f$ and $g$. While our results resolve \neref{eq:moregeneral} whenever $f$ is a polynomial of $k + x$ and $g(k, x) = k + x$, they do not for more general $f$ and $g$, at least not obviously.

\section{Proofs of Results}
\label{sec:mainresults}

In this section we prove Propositions~\ref{prop:marichalzenaidizetaprop2} and \ref{prop:phibohrpolynomial} as well as \thmrefs{thm:nonalternatingcase}{thm:alternatingcase}. These results rely on a handful of intermediate lemmas. Note that Lemmas \ref{lem:shiftzeta}, \ref{lem:BernoulliRecurrence}, and \ref{lem:EulerRecurrence} are trivial, so we deliberately forgo their proofs.

\begin{lemma} 
\label{lem:shiftzeta}
For all $x \in \R_+ \cup \{0\}$, $k \in \N$, and $s \in \C$ with $\Re(s) > 1$, the Hurwitz zeta function satisfies the shift property
$$
\zeta(s, k + x) \defeq \sum_{n \geq 0}\frac{1}{(n + k + x)^s} = \sum_{n \geq k}\frac{1}{(n + x)^s}.
$$
\end{lemma}

The following is a generalization of both Propositions~\ref{prop:marichalzenaidizetaprop} and \ref{prop:marichalzenaidizetaprop2}. In particular, \propref{prop:marichalzenaidizetaprop2} follows with $a = 0$ and $\omega = 1$.
\begin{lemma}
\label{lem:phibohr}
Fix $\omega \in \R$ with $\omega \neq 0$. If $\omega > 0$ ($\omega < 0$), then for all real $s > \ell + 2$ and $p \in \{0,1\}$, all eventually $p$-concave ($p$-convex) solutions $f_{s,a} : \R_+ \rightarrow \R$ to the difference equation $\Delta f_{s,a}(x) = \omega x^a\zeta(s,x)$ are of the form
\begin{equation}
f_{s,a}(x) = c_{s,a} - \omega x^a\zeta(s,x) + \omega\sumvariable_+(s,a,0) - \omega\sumvariable_+(s,a,x)
\label{eq:motivationforPhi}
\end{equation}
for some constant $c_{s,a} \in \R$.
\end{lemma}
\begin{proof}
Fix $a \in \N$, $s > a + 2$, and $p \in \{0,1\}$, and suppose $\omega > 0$ ($\omega < 0$). Then the function $\omega x^a\zeta(s,x)$ is $p$-convex ($p$-concave) in $x$ on $\R_+$. Moreover, by \lemref{lem:shiftzeta} and the bound $s > a + 2$, the sequence $n \mapsto \omega\Delta^p \big(n^a\zeta(s,n)\big)$ converges to zero as $n \rightarrow \infty$. Thus, if $g_{s,a}(x) \defeq \omega x^a\zeta(s,x)$, then the premises of \thmref{thm:maintheorem} are satisfied. Therefore, \eref{eq:uniquesolution} implies that \eref{eq:motivationforPhi} constitutes the unique, eventually $p$-concave ($p$-convex) solutions $f_{s,a} : \R_+ \rightarrow \R$ to the difference equation $\Delta f_{s,a}(x) = g_{s,a}(x)$, as desired.
\end{proof}

Note, for a real, degree-$\ell$ polynomial $Q(x) = \omega_0 + \omega_1 x + \cdots + \omega_\ell x^\ell$ with $\omega_\ell \neq 0$, if $s > 1$ and $\omega_\ell > 0$ ($\omega_\ell < 0$), then, for all $p \in \{0,1\}$, the function $Q(x)\zeta(s,x)$ is eventually $p$-convex ($p$-concave) in $x$ on $\R_+$. Hence, for fixed $s > \ell + 2$ and all $p \in \{0,1\}$, \thmref{thm:maintheorem} and \lemref{lem:phibohr} imply that the unique, eventually $p$-concave ($p$-convex) solutions $f_{s,Q} : \R_+ \rightarrow \R$ to the difference equation $\Delta f_{s,Q}(x) = Q(x)\zeta(s,x)$ are of the form $f_{s,Q}(x) = \sum_{a = 0}^\ell f_{s,a}(x)$. This is \propref{prop:phibohrpolynomial}.

\begin{lemma}
\label{lem:absoluteconvergence}
If $a \in \N$ and $s \in \C$ with $\Re(s) > a + 2$, then, for all $x \in \R_+ \cup \{0\}$, 
$$
\mathrm{abs}(\sumvariable_{\pm}(s,a,x)) \defeq \sum_{k \geq 1} \sum_{n \geq 0} \left|\frac{(k + x)^a}{(n + k + x)^s}\right| < \infty.
$$ 
\end{lemma}

\begin{proof}
Fix $x \in \R_+ \cup \{0\}$. As $a \geq 0$, $\mathrm{abs}(\sumvariable_{\pm}(s,a,x))$ satisfies
$$
\mathrm{abs}(\sumvariable_{\pm}(s,a,x)) \leq \sum_{k \geq 1} \sum_{n \geq0} \frac{1}{(n + k + x)^{\Re(s) - a}}.
$$
Moreover, since $\Re(s) > a + 2$, there exists $\epsilon > 0$ such that $\Re(s) - a \geq 2 + \epsilon$. Therefore,
\begin{align*}
\mathrm{abs}(\sumvariable_{\pm}(s,a,x)) &\leq \sum_{k \geq 1} \sum_{n \geq0} \frac{1}{(n + k + x)^{2 + \epsilon}}\\
&= \sum_{k \geq 1} \zeta(2 + \epsilon, k + x).
\end{align*}
As $x \geq 0$, $\zeta(2 + \epsilon, k + x) \leq \zeta(2 + \epsilon, k)$ for every $k \geq 1$. Hence, 
$$
\mathrm{abs}(\sumvariable_{\pm}(s,a,x)) \leq \sum_{k \geq 1} \zeta(2 + \epsilon, k) = \zeta(1 + \epsilon),
$$
where the equality follows from \eref{eq:zetaidentity}. Since $\epsilon > 0$, the RHS is finite.
\end{proof}

Therefore, under the conditions of \lemref{lem:absoluteconvergence}, we can permute the terms in $\sumvariable_{\pm}(s,a,x)$ with impunity, including the terms that arise from different Hurwitz zeta functions. In particular, Lemmas \ref{lem:shiftzeta} and \ref{lem:absoluteconvergence} imply that $\sumvariable_{\pm}(s,a,x)$ satisfies
\begin{equation}
\sumvariable_{\pm}(s,a,x) = \sum_{k \geq 1}(\pm 1)^k\sum_{n \geq k}\frac{(k + x)^a}{(n + x)^s}
\label{eq:shiftofphi}
\end{equation}
whenever $a \in \N$, $x \in \R_+ \cup \{0\}$, and $s \in \C$ with $\Re(s) > a + 2$.

\begin{lemma}
\label{lem:BernoulliRecurrence}
The Bernoulli polynomials satisfy $B_a(x + 1) = ax ^{a-1} + B_a(x)$ for all $a \in \N$ and $x \in \R$ \cite{Abramowitz1, Apostol1}. Therefore, if $k \in \N$, then
$$
B_a(k + x + 1) = B_a(x) + a\sum_{j = 0}^{k} (j + x)^{a-1}.
$$
\end{lemma}

\begin{lemma}
\label{lem:EulerRecurrence}
The Euler polynomials satisfy $E_a(x + 1) = 2x^a - E_a(x)$ for all $a \in \N$ and $x \in \R$ \cite{Abramowitz1}. Therefore, if $k \in \N$, then
$$
(-1)^kE_{a}(k + x + 1) = E_a(x + 1) + 2\sum_{j=1}^k (-1)^j (j + x)^a.
$$
\end{lemma}

\begin{lemma}
\label{lem:DirichletSeriesBernoulli}
Fix $a \in \N$ and $s \in \C$ with $\Re(s) > a + 2$. Then, for all $x \in \R_+ \cup \{0\}$, the Dirichlet-Hurwitz series of the Bernoulli polynomials satisfies
$$
\sum_{k \geq 1} \frac{B_{a}(k + x)}{(k + x)^{s}} = -a \zeta(s-a+1,x+1) + \sum_{j=0}^a \binom{a}{j} B_j \zeta(s-a+j,x+1).
$$
\end{lemma}
\begin{proof}
Fix $x \in \R_+ \cup \{0\}$. With $B_1 = +\frac{1}{2}$, the Bernoulli polynomials satisfy \cite{Abramowitz1, Apostol1}
$$
B_a(x) = - ax^{a-1} + \sum_{j = 0}^a \binom{a}{j} B_j x^{a-j}.
$$ 
Therefore,
\begin{align*}
\sum_{k \geq 1} \frac{B_{a}(k+x)}{(k+x)^{s}} &= \sum_{k \geq 1} \frac{1}{(k+x)^s} \left( - a (k+x)^{a-1} + \sum_{j = 0}^a \binom{a}{j} B_j (k+x)^{a-j} \right) \\
&= -a \zeta(s-a+1,x+1) + \sum_{j=0}^a \binom{a}{j} B_j \zeta(s-a+j,x+1),
\end{align*}
where the rearrangement is justified because $\Re(s) > a + 2$, so the series over the Bernoulli polynomials absolutely converges.
\end{proof}

\begin{lemma}
\label{lem:DirichletSeriesEuler}
Fix $a \in \N$ and $s \in \C$ with $\Re(s) > a + 2$. Then, for all $x \in \R_+ \cup \{0\}$, the alternating Dirichlet-Hurwitz series of the Euler polynomials satisfies 
$$
\sum_{k \geq 1} \frac{(-1)^{k+1}E_{a}(k + x + 1)}{(k + x)^{s}} = \sum_{j=0}^{a} \sum_{\ell = 0}^{a-j} \frac{E_j}{2^{a-\ell}} 
 \binom{a}{j} \binom{a - j}{\ell} L(-1, s - \ell, x+1).
$$
\end{lemma}
\begin{proof}
The Euler polynomials satisfy \cite{Abramowitz1}
\begin{align*}
E_a(x + 1) &= \sum_{j=0}^a \binom{a}{j} \frac{E_j}{2^j} \left( x + \frac{1}{2} \right)^{a-j}\\
&= \sum_{j=0}^a \binom{a}{j} \frac{E_j}{2^j} \sum_{\ell = 0}^{a-j} \binom{a-j}{\ell} \frac{x^\ell}{2^{a-j-\ell}},
\end{align*}
where the equality in the second line follows from the binomial theorem. Therefore,
\begin{align*}
\sum_{k \geq 1} \frac{(-1)^{k+1}E_{a}(k + x + 1)}{(k + x)^{s}} &= \sum_{k \geq 1} \frac{(-1)^{k+1}}{(k + x)^{s}} \sum_{j=0}^{a} \binom{a}{j} \frac{E_j}{2^j} \sum_{\ell = 0}^{a-j} \binom{a - j}{\ell} \frac{(k + x)^\ell}{2^{a-j-\ell}}\\
&=\sum_{j=0}^{a} \sum_{\ell = 0}^{a-j} \frac{E_j}{2^{a-\ell}} \binom{a}{j} \binom{a - j}{\ell} \sum_{k \geq 1} \frac{(-1)^{k+1}}{(k + x)^{s-\ell}},
\end{align*}
where the rearrangement in the second line is justified because $\Re(s) > a + 2$, so the series over the Euler polynomials absolutely converges. Using the definition \neref{eq:Lerch} of the Lerch zeta function, it is straightforward to prove that the internal sum over $k$ equals $L(-1, s - \ell, x + 1)$, as desired.
\end{proof}

\subsection*{Proof of \thmref{thm:nonalternatingcase}}

Fix $x \in \R_+ \cup \{0\}$ and consider the quantity $B_a(x + 1)\zeta(s,x + 1) + a\sumvariable_+(s, a-1,x)$, where $a-1 \in \N$ and $\Re(s) > a+1$. The premises of \lemref{lem:absoluteconvergence} hold, so we are allowed to permute terms with impunity. In particular, by \eref{eq:shiftofphi},
$$
B_a(x + 1) \zeta(s, x+1) + a\sumvariable_+(s, a-1, x) = \sum_{k \geq 1} \frac{B_a(x + 1)}{{(k + x)^s}} + \sum_{k \geq 1}\sum_{n \geq k}\frac{a(k + x)^{a-1}}{(n + x)^s}.
$$
Combining terms with like-denominators gives
$$
B_a(x + 1) \zeta(s, x+1) + a\sumvariable_+(s, a-1, x) = \sum_{k \geq 1} \frac{B_a(x + 1) + a\sum_{j = 1}^k (j + x)^{a-1}}{(k + x)^s}.
$$
Then, by the Bernoulli polynomial recurrence $B_a(x + 1) = B_a(x) + ax^{a-1}$,
$$
B_a(x + 1) \zeta(s, x+1) + a\sumvariable_+(s, a-1, x) = \sum_{k \geq 1} \frac{B_a(x) + a\sum_{j = 0}^k (j + x)^{a-1}}{(k + x)^s}.
$$
By \lemref{lem:BernoulliRecurrence}, we identify the numerator as $B_a(k + x + 1)$. Therefore,
$$
B_a(x + 1) \zeta(s, x+1) + a\sumvariable_+(s, a-1,x) = \sum_{k \geq 1}\frac{B_a(k + x + 1)}{(k + x)^s},
$$
which, after a further application of the recurrence $B_a(x + 1) = B_a(x) + ax^{a-1}$, yields
$$
B_a(x + 1) \zeta(s, x+1) + a\sumvariable_+(s, a-1,x) = a \zeta(s - a + 1, x + 1) + \sum_{k \geq 1}\frac{B_a(k + x)}{(k + x)^s}.
$$
Consequently, $\sumvariable_+(s,a-1,x)$ is related to the Dirichlet-Hurwitz series of the Bernoulli polynomials. It thus follows from \lemref{lem:DirichletSeriesBernoulli} that
$$
\sumvariable_+(s, a-1,x) = -\frac{1}{a}B_a(x+1)\zeta(s,x+1) + \frac{1}{a}\sum_{j=0}^a \binom{a}{j} B_j \zeta(s-a+j,x+1).
$$
Taking $a \rightarrow a + 1$ gives \thmref{thm:nonalternatingcase}.\qed

\subsection*{Proof of \thmref{thm:alternatingcase}}
Fix $x \in \R_+ \cup \{0\}$ and consider the quantity $E_a(x + 1) \zeta(s, x+1) - 2\sumvariable_{-}(s,a, x)$, where $a \in \N$ and $\Re(s) > a+2$. The premises of \lemref{lem:absoluteconvergence} hold, so we are allowed to permute terms with impunity. In particular, by \eref{eq:shiftofphi},
$$
E_a(x + 1) \zeta(s, x+1) - 2\sumvariable_{-}(s,a, x) = \sum_{k \geq 1} \frac{E_a(x+1)}{(k + x)^s} - \sum_{k \geq 1}(\pm 1)^{k}\sum_{n \geq k}\frac{2(k + x)^a}{(n + x)^s}.
$$
Combining terms with like-denominators gives
$$
E_a(x + 1) \zeta(s, x+1) - 2\sumvariable_{-}(s,a, x) = \sum_{k \geq 1} \frac{E_a(x + 1) + 2\sum_{j = 1}^k (-1)^{j + 1}(j + x)^a}{(k + x)^s}.
$$
By \lemref{lem:EulerRecurrence}, the terms in the numerator are related to the Euler polynomials. In particular,
$$
E_a(x + 1) + 2\sum_{j=1}^k (-1)^{j + 1} (j + x)^a = (-1)^{k + 1}E_{a}(k + x + 1) + 2E_a(x + 1).
$$
Therefore,
$$
E_a(x + 1) \zeta(s, x+1) - 2\sumvariable_{-}(s,a, x) = \sum_{k \geq 1} \frac{(-1)^{k + 1}E_{a}(k + x + 1)}{(k + x)^s} + 2E_a(x+1)\zeta(s,x+1).
$$
Consequently, $\sumvariable_{-}(s,a,x)$ is related to the alternating Dirichlet-Hurwitz series of the Euler polynomials. It thus follows from \lemref{lem:DirichletSeriesEuler} that
$$
\sumvariable_{-}(s,a, x) = - \frac{1}{2}E_a(x+1)\zeta(s,x+1) - \sum_{j=0}^{a} \sum_{\ell = 0}^{a-j} \frac{E_j}{2^{a-\ell + 1}}\binom{a}{j} \binom{a - j}{\ell} L(-1, s - \ell, x+1),
$$
as desired. \qed

%% Begin Bibliography %%
\bibliographystyle{amsplain}
\bibliography{Hurwitz}

\providecommand{\bysame}{\leavevmode\hbox to3em{\hrulefill}\thinspace}
\providecommand{\MR}{\relax\ifhmode\unskip\space\fi MR }
\providecommand{\MRhref}[2]{%
  \href{http://www.ams.org/mathscinet-getitem?mr=#1}{#2}
}
\providecommand{\href}[2]{#2}
\begin{thebibliography}{10}

\bibitem{Abramowitz1}
M.~Abramowitz and I.~A. Stegun, \emph{Handbook of mathematical functions: with
  formulas, graphs, and mathematical tables}, Dover, New York, 1965.

\bibitem{Apostol1}
T.~M. Apostol, \emph{Introduction to analytic number theory}, Springer, New
  York, 1976.

\bibitem{Artin1}
E.~Artin, \emph{The gamma function}, Holt, Rinehart \& Winston, Austin, 1964.

\bibitem{Bohr1}
H.~Bohr and J.~Mollerup, \emph{L{\ae}rebog i kompleks analyse}, vol. III,
  Gjellerups Forlag, Copenhagen, 1922.

\bibitem{Boor1}
C.~{de Boor}, \emph{A practical guide to splines}, revised ed.,
  Springer-Verlag, New York, 2001.

\bibitem{Espinosa1}
O.~Espinosa and V.~H. Moll, \emph{On some integrals involving the {H}urwitz
  zeta function: Part 1}, Ramanujan J. \textbf{6} (2002), 159--188.

\bibitem{Espinosa2}
\bysame, \emph{On some integrals involving the {H}urwitz zeta function: Part
  2}, Ramanujan J. \textbf{6} (2002), 449--468.

\bibitem{Krull1}
W.~Krull, \emph{Bemerkungen zur differenzengleichung $g(x+1) - g(x)=\phi(x)$},
  Math. Nachr. \textbf{1} (1948), 365–376.

\bibitem{Krull2}
\bysame, \emph{Bemerkungen zur differenzengleichung $g(x+1) - g(x)=\phi(x)$.
  {II}.}, Math. Nachr. \textbf{2} (1949), 251--262.

\bibitem{Kuczma1}
M.~Kuczma, \emph{An introdution to the theory of functional equations and
  inequalities}, second ed., Birkh{\"a}user-Verlag, Berlin, 2009.

\bibitem{Marichal1}
J.~L. Marichal and N.~Zena{\"i}di, \emph{A generalization of
  {B}ohr-{M}ollerup's theorem for higher order convex functions}, Springer,
  Cham, Switzerland, 2022.

\bibitem{Paris1}
R.~B. Paris, \emph{A note on some infinite sums of {H}urwitz zeta functions},
  arXiv:2104.00957v2 (2021), 1--9.

\bibitem{Webster2}
R.~Webster, \emph{Log-convex solutions to the functional equation
  $f(x+1)=g(x)f(x)$: {$\Gamma$}-type functions}, J. Math. Anal. Appl.
  \textbf{209} (1997), 605--623.

\bibitem{Webster1}
\bysame, \emph{On the {B}ohr-{M}ollerup-{A}rtin characterization of the gamma
  function}, Rev. Anal. Num{\'e}r. Th{\'e}or. Approx. \textbf{26} (1997),
  249--258.

\end{thebibliography}

\end{document}